\documentclass[12pt]{amsart}
\usepackage{amssymb}
\usepackage{amsmath}
\usepackage{amsthm}
\usepackage{amsbsy}
\usepackage{psfrag}
\usepackage{pstricks}
\usepackage{graphics}
\usepackage{graphicx}
\theoremstyle{plain}
\newtheorem{theorem}{Theorem}[section]
\newtheorem{lemma}[theorem]{Lemma}
\newtheorem{corollary}[theorem]{Corollary}

\newtheorem{proposition}[theorem]{Proposition}
\theoremstyle{remark}
\newtheorem{remark}[theorem]{Remark}

\begin{document}
\allowdisplaybreaks[4]
\numberwithin{figure}{section}
\numberwithin{table}{section}
 \numberwithin{equation}{section}
% \numberwithin{figure}{section}
%
\title[A Posteriori Error Estimator for the Obstacle Problem]
 {Inhomogeneous Dirichlet Boundary Condition in the A Posteriori Error Control of the Obstacle Problem}

\author{Sharat Gaddam}
\address{Department of Mathematics, Indian Institute of Science, Bangalore - 560012}
\email{sharat12@math.iisc.ernet.in}

\author{Thirupathi Gudi}
\address{Department of Mathematics, Indian Institute of Science, Bangalore - 560012}
\email{gudi@math.iisc.ernet.in}

\date{}
\begin{abstract}
We propose a new and simpler residual based a posteriori error
estimator for finite element approximation of the elliptic
obstacle problem. The results in the article are two fold.
Firstly, we address the influence of the inhomogeneous Dirichlet
boundary condition in {\em a posteriori} error control of the
elliptic obstacle problem. Secondly by rewriting the obstacle
problem in an equivalent form, we derive simpler {\em a
posteriori} error bounds which are free from min/max functions. To
accomplish this, we construct a post-processed solution $\tilde
u_h$ of the discrete solution $u_h$ which satisfies the exact
boundary conditions although the discrete solution $u_h$ may not
satisfy. We propose two post processing methods and analyze them.
We remark that the results known in the literature are either for
the homogeneous Dirichlet boundary condition or that the estimator
is only weakly reliable in the case of inhomogeneous Dirichlet
boundary condition.

\end{abstract}
\keywords{finite element, a posteriori error estimate, obstacle
problem, variational inequalities, Lagrange multiplier, reliable,
efficient}
\subjclass{65N30, 65N15}
\maketitle
%%%%%%%%%%%%%%%%%%%%%%%%%%%%%%%%%%%%%%%%%%%%%%%%%%%%%%%%%%%%%%%%
\allowdisplaybreaks
\def\R{\mathbb{R}}
\def\cA{\mathcal{A}}
\def\cK{\mathcal{K}}
\def\cN{\mathcal{N}}
\def\p{\partial}
\def\O{\Omega}
\def\bbP{\mathbb{P}}
\def\cV{\mathcal{V}}
\def\cR{\mathcal{R}}
\def\cT{\mathcal{T}}
\def\cE{\mathcal{E}}
\def\bF{\mathbb{F}}
\def\bC{\mathbb{C}}
\def\bN{\mathbb{N}}
\def\ssT{{\scriptscriptstyle T}}
\def\HT{{H^2(\O,\cT_h)}}
\def\mean#1{\left\{\hskip -5pt\left\{#1\right\}\hskip -5pt\right\}}
\def\jump#1{\left[\hskip -3.5pt\left[#1\right]\hskip -3.5pt\right]}
\def\smean#1{\{\hskip -3pt\{#1\}\hskip -3pt\}}
\def\sjump#1{[\hskip -1.5pt[#1]\hskip -1.5pt]}
\def\jumptwo{\jump{\frac{\p^2 u_h}{\p n^2}}}
%%%%%%%%%%%%%%%%%%%%%%%%%%%%%%%%%%%%%%%%%%%%%%%%%%%%%%%%%%%%%%%%%%%%%%%%%%%
%%%%%%%%%%%%%%%%%%%%%%%%%%%%%%%%%%%%%%%%%%%%%%%%%%%%%%%%%%%%%%%%%
\section{Introduction}\label{sec:Intro}
The elliptic obstacle problem is one of the popular prototype
models for the study of elliptic variational inequalities. The
applications of variational inequalities are enormous in the
modern scientific computing world, e.g. in contact mechanics,
option pricing and fluid flow problems. The numerical analysis of
these class of problems is an interesting subject as they offer
challenges both in theory and computation. We refer to the books
\cite{AH:2009:VI,Glowinski:2008:VI,KS:2000:VI,Sutt:book} for the
theory of variational inequalities and their corresponding
numerical analysis. Apart from these, we refer to the articles
\cite{BHR:1977:VI,Falk:1974:VI} and the recent articles
\cite{BLY:2012:C0IP,BLY:2012:C0IP1,HK:1994:multiadaptive,Wang:2002:P2VI,WHC:2010:DGVI}
for the convergence analysis of finite element methods for the
obstacle problem. The obstacle problem exhibits free boundaries
where the regularity of the solution is affected. It is worth
remarking here that the location of a free boundary is not known
{\em a priori}. Adaptive finite element methods based on reliable
and efficient {\em a posteriori} error estimates are of particular
interest in this contest as they can capture the free boundaries
by local mesh refinement around them.  In designing any of
adaptive schemes, the first step is to derive some computable
error estimators which are both reliable and efficient, see
\cite{AO:2000:Book} for error analysis of various type. There are
many works in deriving residual based a posteriori error estimates
for the obstacle problem, see
\cite{AOL:1993:Apost,BC:2004:VI,Braess:2005:VI,CN:2000:VI,GKVZ:2011:VeeserHirarchy,TG:2014:VIDG,TG:2014:VIDG1,TG:2015:VIP2,NSV:2005:Apost,NPZ:2010:VI,Veeser:2001:VI,WW:2010:Apost}
and see
\cite{BS:2014:hp-apost,Gwinner:2009:pfem,Gwinner:2013:pfem,WHE:2015:ApostDG}.
In recent years, much of research is focused on proving the
convergence of adaptive methods based on a posteriori error
estimates. In this direction, we refer to
\cite{BCH:2007:VI,BCH:2009:AVI,FPP:2014:Obstacle,SV:2007:APost,PP:2013:AconvVI}
for the work related to the obstacle problem. Further, we refer
\cite{Belgacem:2000:SINUM,DH:2015:Signorini,HR:2012:Signorini,HW:2005:Signorini}
and
\cite{AS:2011:Contact,BS:2000:VI,HN:2005:Signorini,WW:2009:Contact}
for the work related to the numerical approximation of the
Signorini contact problem.

In many occasions, it is assumed for convenience in the a
posteriori error analysis of obstacle problems that the Dirichlet
data is either zero or trace of a finite element function. However
it is not clear if the error estimator with homogeneous Dirichlet
boundary condition is reliable and efficient in the energy norm up
to some Dirichlet data oscillations. The answer to this question
so far seems to be not clear as it can be seen in \cite[Section
4.2]{FPP:2014:Obstacle} that the error estimator is proved only
weakly reliable with nonhomogeneous Dirichlet boundary data
oscillations. Exceptions to this question hold for the local error
analysis (estimates in maximum norm) in
\cite{NSV:2003:Local,NSV:2005:Apost} where  reliable and efficient
error estimates were derived for general Dirichlet boundary data.
One of the difficulties in the energy norm error estimate arises
due to the fact that the error $u-u_h$ does not belong to
$H^1_0(\Omega)$. It will be difficult to argue with the residual
directly using the error $u-u_h$. We may think of introducing an
auxiliary problem correcting this as in the case of linear
elliptic problems, but with this the estimator will consists of
the unknown solution of that auxiliary problem. In this article,
we introduce two options of constructing a post-processed solution
in which one of them can be computed explicitly. We also derive
some estimates for the discrete solution and its post-processed
solution. Thereby, we address the question of proving the
reliability and efficiency of the error estimator up to some
Dirichlet data oscillations as in the case of linear elliptic
problems. The results also can be viewed in another aspect that
since the obstacle problem with a general obstacle $\chi$ can be
transformed into a problem with zero obstacle with nonhomogeneous
Dirichlet boundary condition, for example see
\cite{FPP:2014:Obstacle},  the error estimator for  general
obstacle problem can be simplified to the error estimator for the
zero obstacle. Generally, the error estimator for zero obstacle
problem is simpler and is free of min/max functions dealing with
inconsistency of obstacle constraint.

The rest of the article is organized as follows. In the remaining
part of this section, we introduce the model problem and some
preliminaries. In section \ref{sec:Prelims}, we introduce some
notation, define the discrete problem and derive some properties
of the discrete solution. In section \ref{sec:post-solu}, we
construct a post-processed  discrete solution. We propose two
methods for this purpose. One of them is by  harmonic extension
and the other by linear extension. Therein, we derive some error
estimates for the discrete solution and its post-processed
solution. In section \ref{sec:ErrorAnalysis}, we present the a
posteriori error analysis. In section \ref{sec:Simplified}, we
derive error bounds that are simpler by rewriting general obstacle
problem into a problem with zero obstacle.  We present some
numerical experiments in section \ref{sec:Numerics} to illustrate
the theoretical findings and finally conclude the article in
section \ref{sec:conclusions}.

Let $\Omega\in \R^2$ be a bounded polygonal domain  with boundary
denoted by $\partial\Omega$ (without slit). However the results on
harmonic extension and the results in section 5 are applicable to
three dimensional problems. We consider the obstacle $\chi \in
C(\bar\Omega)\cap H^1(\Omega)$ satisfying
$\chi|_{\partial\Omega}\leq g$ a.e. on $\partial\Omega$, hereafter
the function $g$ is assumed to be given and denotes the Dirichlet
boundary data. Further, we assume that $g$ is the trace of a
$H^1(\Omega)$ function. Define the closed and convex set by
\begin{align*}
\cK :=\{v\in H^1(\Omega): v\geq \chi \text{ a.e. in } \Omega,\;\;
\gamma_0(v)=g \text{ on } \partial\Omega \},
\end{align*}
where $\gamma_0:H^1(\Omega)\rightarrow  L^2 (\partial\Omega)$ is
the trace map, whose range is denoted by $\tilde
H^{1/2}(\partial\Omega)$. Since $g\in \tilde
H^{1/2}(\partial\Omega)$, there is some $\tilde g\in H^1(\Omega)$
with $\gamma_0(\tilde g)=g$. Then it can be seen that the set
$\cK$ is nonempty as $\chi^+ :=\max\{\chi,\tilde g\}\in \cK$. The
model problem for the discussion consists of finding $u\in \cK$
such that
\begin{align}\label{eq:MP}
a(u,v-u)\geq (f,v-u)\;\; \text{ for all } v\in \cK,
\end{align}
where $a(u,v):=(\nabla u,\nabla v)$ and $f\in L^2(\Omega)$ is a
given function. Hereafter, $(\cdot,\cdot)$ denotes the
$L^2(\Omega)$ inner-product while $\|\cdot\|$ denotes the
$L^2(\Omega)$ norm. The existence of a unique solution to
\eqref{eq:MP} follows from the result of Stampacchia
\cite{AH:2009:VI,Glowinski:2008:VI,KS:2000:VI}.

For the rest of the discussions, we assume that the Dirichlet data
$g\in \tilde H^{1/2}(\partial\Omega)\cap C(\partial\Omega)$.

\par
\noindent Define the Lagrange multiplier
 $\sigma\in H^{-1}(\Omega)$  by
\begin{align}\label{eq:sigmadef}
\langle \sigma, v\rangle =(f,v)-a(u,v)\;\;\text{ for all } v\in
H^1_0(\Omega),
\end{align}
where $\langle \cdot,\cdot\rangle$ denotes the duality bracket of
$H^{-1}(\Omega)$ and $H^1_0(\Omega)$. It follows from
\eqref{eq:sigmadef} and \eqref{eq:MP} that
\begin{equation}\label{eq:sigma}
\langle \sigma, v-u\rangle\leq 0\; \text{ for all } v \in \cK.
\end{equation}
\par
\noindent Note that in the above equation \eqref{eq:sigma}, the
test function $v-u\in H^1_0(\Omega)$ and hence the duality bracket
$\langle \sigma, v-u\rangle$ is meaningful.

\section{Notation and Preliminaries}\label{sec:Prelims}
Below, we list the notation that will be used throughout the
article:
\begin{align*}
\cT_h &:=\text{a regular simplicial triangulations of } \Omega\\
T&:=\text{a triangle of } \cT_h,\qquad |T|:=\text{ area of } T\\
 h_T &:=\text{diameter of } T, \qquad h :=\max\{h_T : T\in\cT_h\}\\
\cV_h^i& :=\text{set of all vertices of } \cT_h \text{ that are in
}
\O\\
\cV_h^b& :=\text{set of all vertices of } \cT_h \text{ that are on
}
\p\O\\
\cV_T& :=\text{set of three vertices of } T\\
\cE_h^i&:=\text{set of all interior edges of } \cT_h\\
\cE_h^b&:=\text{set of all boundary edges of } \cT_h\\
h_e&:=\text{length of an edge } e\in\cE_h.
\end{align*}
We mean by regular triangulation that the triangles in $\cT_h$ are
shape-regular and there are no hanging nodes in $\cT_h$. Each
triangle $T$ in $\cT_h$ is assumed to be closed. Define $\cT_h^i$
to be the set of all triangles which do not share an edge with
boundary $\p\O$. Let $\cT_h^b$ denote the remaining set of
triangles, i.e., the set of all triangles $T\in \cT_h$ which have
an edge $e\in \cE_h^b$ on its boundary $\p T$. For simplicity
assume that any $T\in \cT_h^b$ shares at most one edge with
boundary $\p\Omega$, otherwise the triangle $T$ has no interior
node.

In order to define the jump of discontinuous functions
conveniently, define a broken Sobolev space
\begin{eqnarray*}
H^1(\O,\cT_h) :=\{v\in L^2(\Omega) :\,v_{\ssT}= v|_{T}\in
H^1(T)\,\text{ for all }\,~T\in{\mathcal T}_h\}.
\end{eqnarray*}
For any $e\in\cE_h^i$, there are two triangles $T_+$ and $T_-$
such that $e=\partial T_+\cap\partial T_-$. Let $n_+$ be the unit
normal of $e$ pointing from $T_+$ to $T_-$, and $n_-=-n_+$, see
Fig \ref{Fig1}. For any $v\in H^1(\Omega,{\mathcal T}_h)$, we
define the jump of $v$ on $e$ by
\begin{eqnarray*}
 \sjump{v} : = v_+ + v_-,
\end{eqnarray*}
where $v_\pm=v\big|_{T_\pm}$. Similarly define the jump of $w\in
H^1(\Omega,{\mathcal T}_h)^2$ on $e\in\cE_h^i$ by
\begin{eqnarray*}
\sjump{w} := w_-\cdot n_-+w_+\cdot n_+,
\end{eqnarray*}
where $w_\pm=w |_{T_\pm}$. For any edge $e\in \cE_h^b$,  there is
a triangle $T\in\cT_h$ such that $e=\partial T\cap
\partial\Omega$. For any $v\in H^1(T)$, we set on $e\in \cE_h^b$
\begin{eqnarray*}
 \sjump{v} := v.
\end{eqnarray*}

\subsection{Discrete Problem}
We use the following linear finite element spaces $V_h$ and
$V_h^0$ defined on $\cT_h$ by
$$V_h :=\{v_h\in H^1(\Omega): v_h|_{T} \in \bbP_1(T)  \text{ for all } T\in\cT_h\},$$
and $V_h^0=V_h\cap H^1_0(\O),$ respectively. Define the discrete
trace space $W_h$ by
$$W_h :=\{w_h\in C(\partial\Omega): v_h|_{e} \in \bbP_1(e)  \text{ for all } e\in\cE_h^b\}.$$
Let $g_h\in W_h$ be the nodal interpolation of $g$, i.e.
$g_h(z)=g(z)$ for all $z\in \cV_h^b$. Define the discrete set
\begin{align}\label{eq:Kh}
\cK_h :=\{v_h\in V_h: v_h(z)\geq \chi(z)\,\text{ for all }
z\in\cV_h^i, \quad v_h=g_h\text{ on } \p\O \}.
\end{align}
The discrete problem consists of finding $u_h\in\cK_h$ such that
\begin{align}\label{eq:DP}
a(u_h,v_h-u_h)\geq (f,v_h-u_h)\,\text{ for all } v_h\in \cK_h.
\end{align}
Similar to the case of the continuous problem \eqref{eq:MP}, it
can easily be shown that the discrete problem \eqref{eq:DP} has a
unique solution. Below, we derive some properties of the solution
$u_h$.
\par
Note that for any $z\in \cV_h^i$ and the corresponding Lagrange
basis function (hat function) $\psi_z$ defined by
\begin{equation*}
q\in \cV_h^i, \qquad \psi_z (q) := \left\{ \begin{array}{ll} 1 & \text{ if } z = q,\\\\
0 & \text{otherwise},
\end{array}\right.
\end{equation*}
we have $v_h=u_h+\psi_z \in \cK_h$. Then \eqref{eq:DP} implies
that
\begin{align}\label{eq:DPProperty1}
a(u_h,\psi_z)\geq (f,\psi_z) \;\; \text{ for all }\; z \in
\cV_h^i.
\end{align}
Furthermore if $u_h(z)>\chi(z)$ for any $z\in\cV_h^i$, it is easy
to prove  by considering $v_h^\pm=u_h\pm \epsilon \psi_z \in
\cK_h$ (for sufficiently small $\epsilon>0$) in \eqref{eq:DP} that
\begin{align}\label{eq:DPProperty2}
a(u_h,\psi_z)=(f,\psi_z) \;\; \text{ for all }\; z \in \{z
\in\cV_h^i :u_h(z)>\chi(z)\}.
\end{align}
\par
\noindent As in the case of continuous problem, we define the
discrete Lagrange multiplier $\sigma_h \in V_h^0$ by
\begin{align}\label{eq:sigmah-def}
\langle \sigma_h,v \rangle_h = (f, v)-a(u_h,v) \quad \forall
\,v\in V_h^0,
\end{align}
where
\begin{align}\label{eq:discret-IP}
\langle w,v \rangle_h =\sum_{T\in \cT_h}\frac{|T|}{3}\sum_{z\in
\cV_T} w(z)\,v(z),
\end{align}
which is the well-known mass lumping formula.

\par
\noindent Using \eqref{eq:DPProperty1}-\eqref{eq:DPProperty2}, we
deduce that
\begin{align}
\sigma_h(z)&\geq 0 \quad \forall z\in\cV_h,\\
\sigma_h(z)&=0\quad\text{ if } u_h(z)>\chi(z),\;\;z\in\cV_h^i.
\end{align}

\par
\noindent The following error estimate for the lumped-mass
numerical integration is well-known,
\cite{Veeser:2001:VI,TG:2014:VIDG}: For any $T\in\cT_h$ and
$v_h,w_h\in V_h$, there holds
\begin{align}\label{eq:NumerQuad}
\left|\frac{|T|}{3}\sum_{z\in
\cV_T}v_h(z)w_h(z)-\int_Tv_hw_h\,dx\right|\leq C h_T\|\nabla
v_h\|_{L_2(T)}\|\nabla w_h\|_{L_2(T)}.
\end{align}

\section{Post-Processed Solution}\label{sec:post-solu}
Since the discrete problem \eqref{eq:DP} uses the approximate
Dirichlet data $g_h$ which is not equal to the given data $g$, the
numerical scheme can be treated as a nonconforming scheme with
respect to the Dirichlet data. In this section, we propose two
post-processing methods to construct an intermediate solution
$\tilde u_h$ (which can be treated as a post-processing of $u_h$)
constructed with the help of $u_h$ and $g$ as follows.  For
$T\in\cT_h^i$, define $\tilde u_h$ to be the same as $u_h$ on $T$.
For $T\in\cT_h^b$, first define $u_h^*$ on $\p T$ ( the boundary
of $T$) by
\begin{equation}\label{eq:uht}
u_h^*(x,y) := \left\{ \begin{array}{ll} g(x,y) & \text{ if } (x,y)\in \p\O,\\\\
u_h(x,y) & \text{ if } (x,y) \in \p T\backslash \p\O.
\end{array}\right.
\end{equation}
Then define $\tilde u_h$ to be some $H^1(T)$ extension of
$u_h^*\in C(\p T)$ to $T$ such that $\tilde u_h|_{\p T}=u_h^*|_{\p
T}$ and $\tilde u_h|_{T}=u_h|_{T}$ for all $T\in\cT_h^i$. By the
construction $\tilde u_h \in H^1(\Omega)$ and $u_h-\tilde
u_h\equiv 0$ on any triangle $T\in\cT_h^i$.

\par
\noindent In general, there will be many choices of $\tilde u_h$.
In the following we will present two methods for constructing it.
One approach is by using the Harmonic extension and the other is
by linear interpolation. We discuss this in two dimensions context
and note that the similar extension in three dimension is quite
similar.

\par
\noindent {\em Method 1.} ({\em Harmonic Extension}): Let $T\in
\cT_h^b$ and $u_h^*\in C(\p T)$ be defined by  \eqref{eq:uht}.
Then define $\tilde u_h \in H^1(T)$ to be the weak solution of
\begin{align}
\tilde u_h &=u_h^* \text{  on  } \p T,\label{eq:HE1}\\
\Delta \tilde u_h &=0 \text{    in  } T. \label{eq:HE2}
\end{align}
By the stability of elliptic problems and scaling, we find that
\begin{align*}
\|\tilde u_h\|_{H^1(T)}\leq C \|u_h^*\|_{H^{1/2}(\p T)},
\end{align*}
where the constant $C$ is independent of $T$. Note that $\tilde u_h-u_h$ satisfies
\begin{align*}
\tilde u_h-u_h &=u_h^*-u_h \text{  on  } \p T,\\
\Delta (\tilde u_h-u_h) &=0 \text{    in  } T.
\end{align*}
We can show by the stability of elliptic problems and scaling that
\begin{align*}
\|\tilde u_h-u_h\|_{H^1(T)}\leq C \|u_h^*-u_h\|_{H^{1/2}(\p T)}.
\end{align*}
In addition if $g\in H^1(e)$ where $e=\p\O\cap \p T$, then by the
approximation property of Lagrange interpolation, see
\cite{CC:1996:Hhalf}, we find
\begin{align*}
\|u_h^*-u_h\|_{H^{1/2}(\p T)}\leq C h_T^{1/2} |g-g_h|_{H^{1}(e)},
\end{align*}
and hence
\begin{align*}
\|\tilde u_h-u_h\|_{H^1(T)}\leq C h_T^{1/2} |g-g_h|_{H^{1}(e)}.
\end{align*}

\par
\noindent Further since the triangle $T$ is convex and $g\in
H^1(e)$, by the elliptic regularity theory $\tilde u_h\in
H^{3/2}(T)\subset C(\bar T)$. Then the weak maximum principle for
harmonic functions implies that
\begin{align*}
\min_{z\in T}\tilde u_h(z)=\min_{z\in\partial T} u_h^*(z).
\end{align*}
\par
\noindent We deduce the following proposition on the harmonic
extension:
\begin{proposition}
Let $T\in \cT_h^b$ and $e$ be an edge of $T$ with $e\in \cE_h^b$.
Let $\tilde u_h$ be the harmonic extension of $u_h^*$ defined by
$\eqref{eq:HE1}$-$\eqref{eq:HE2}$. Then, if $g\in H^1(e)$, there
holds
\begin{align}\label{eq:OscHg}
\|\tilde u_h-u_h\|_{H^1(T)}\leq C h_T^{1/2} |g-g_h|_{H^{1}(e)},
\end{align}
and  furthermore
\begin{align}\label{eq:minP}
\min_{z\in T}\tilde u_h(z)=\min_{z\in\partial T} u_h^*(z).
\end{align}
\end{proposition}

\par
Next we proceed to present another approach by linear extension
(linear interpolation).

\par
\noindent {\em Method 2.} ({\em Linear Extension}): Let $T\in
\cT_h^b$ and $u_h^*\in C(\p T)$ be defined by  \eqref{eq:uht}. Let
$e\in \cE_h^b$ be such that $e\subset\p T$.  We will construct
$\tilde u_h$ by connecting through line segments using the values
of $u_h^*$ along the lines perpendicular to $e$. Let $z=(x,y)$ be
an interior point of $T$ and $\ell$ be the line that is orthogonal
to $e$ passing through $z$, as it shown in the Figure
\ref{fig:divide-T}. Let $z_1=(x_0,y_0)$ and $z_2=(x_1,y_1)$ be the
points of intersection of $\ell$ with $\p T$. Then $\tilde u_h$ is
defined on $\ell$ as a linear polynomial assuming the values
$u_h^*(x_0,y_0)$ and $u_h^*(x_1,y_1)$ at $z_1=(x_0,y_0)$ and
$z_2=(x_1,y_1)$, respectively.

\begin{figure}[!hh]
\vspace{-.2cm}
\begin{center}
\setlength{\unitlength}{0.7cm}
\begin{pspicture}(8,5)
%\psgrid
%%%%%%%%%%%%%%%%%%%%%%%%%
\put(-1.5,0.8){$(a_1,b_1)$} \put(2.5,4.25){$(a_3,b_3)$}
\put(6.2,0){$(a_2,b_2)$} \qline(3,4)(2.5,0.6)
%(2.108108108108108,0.648648648648649)
\put(1.7,1.2){$\ell$} \put(1,1.5){$T_1$}\put(3.5,1.3){$T_2$}
\qline(0,1)(3,4)\put(2.9,2){$\bar h$} \qline(3,4)(6,0)
\qline(0,1)(6,0) \qline(1.5,0.75)(1.85,2.85)
\put(1.2,0.5){$z_1$}\put(1.5,3){$z_2$} \put(1.85,2){$z$}
\psset{dotsize=4pt} \psdots*(1.708333333,2)(0,1)(3,4)(6,0)
\put(2.5,0.2){$e$}
%%%%%%%%%%%%%%%%%%%%
\end{pspicture}
\vspace{1cm} \caption{Subdivision of $T=T_1\cup T_2$}
\label{fig:divide-T}
\end{center}
\end{figure}
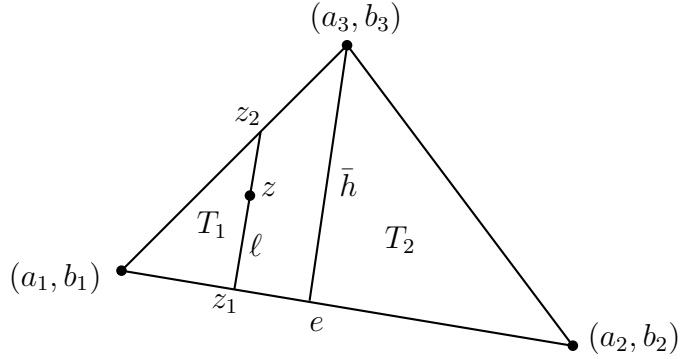
We divide the triangle $T$ as $T=T_1\cup T_2$, where $T_1$ and
$T_2$ are sub triangles of $T$ separated by the height $\bar h$ as
shown in the Figure \ref{fig:divide-T}.
Then define $\tilde u_h$ on $T$ by
\begin{align}\label{eq:LE}
\tilde u_h(x,y) &:=\frac{h_2 u_h^*(x_0,y_0)+h_1
u_h^*(x_1,y_1)}{h_1+h_2}\quad \text{ for } (x,y)\in T^\circ,
\end{align}
where $T^\circ$ is the interior of $T$,

\begin{align}
x_0
&=\frac{(a_1-a_2)\big((a_1-a_2)x+(b_1-b_2)y\big)+(b_1-b_2)(b_1a_2-b_2a_1)}{(a_1-a_2)^2+(b_1-b_2)^2},
\label{eq:x0}\\ \notag\\
y_0
&=\frac{(b_2-b_1)\big((a_2-a_1)x+(b_2-b_1)y\big)+(a_2-a_1)(b_1a_2-b_2a_1)}{(a_1-a_2)^2+(b_1-b_2)^2},
\label{eq:y0}\\ \notag\\
x_1
&=\frac{(a_i-a_3)\big((a_1-a_2)x+(b_1-b_2)y\big)+(b_1-b_2)(b_ia_3-b_3a_i)}{(a_1-a_2)(a_i-a_3)+(b_1-b_2)(b_i-b_3)},
\notag\\ \notag\\
y_1
&=\frac{(b_i-b_3)\big((a_1-a_2)x+(b_1-b_2)y\big)+(a_1-a_2)(b_3a_i-b_ia_3)}{(a_1-a_2)(a_i-a_3)+(b_1-b_2)(b_i-b_3)},
\notag\\ \notag\\
h_1
&=\frac{\mid(b_2-b_1)x+(a_1-a_2)y+(b_1a_2-b_2a_1)\mid}{\sqrt{(b_2-b_1)^2+(a_2-a_1)^2}},
\notag\\ \notag\\
h_2 &=\sqrt{(b_1-b_2)^2+(a_1-a_2)^2}\ \
\frac{\mid(a_i-a_3)y-(b_i-b_3)x+(b_ia_3-b_3a_i)\mid}{\mid(a_1-a_2)(a_i-a_3)+(b_1-b_2)(b_i-b_3)\mid},\notag
\end{align}
where $\mid \cdot\mid$ denotes the absolute value and\\
$i=
\begin{cases}
      1 & \text{if}\;(x,y)\in\;T_1,\\
      2 & \text{if}\;(x,y)\in\;T_2.
\end{cases}
$

\par
\noindent Note that $h_1$ is the distance between $z=(x,y)$ and
$z_1=(x_0,y_0)$; and $h_2$ is the distance between $z=(x,y)$ and
$z_2(x_1,y_1)$.

\par
\noindent
Again by the construction, it is immediate to see that
\begin{align}\label{eq:minP1}
\min_{z\in T}\tilde u_h(z)=\min_{z\in\partial T} u_h^*(z).
\end{align}

\par
\noindent In the rest of this section, we assume that $g|_e\in
H^1(e)$ and we establish the following estimate:
\begin{theorem}\label{thm:Oscg}
Let $T\in\cT_h^b$ and define $\tilde u_h$ by $\eqref{eq:LE}$, a
linear extension of $u_h^*$ given by $\eqref{eq:uht}$. Let $e\in
\cE_h^b$ be an edge of $T$ and  assume that $g|_e\in H^1(e)$. Then
there holds
\begin{align*}
\|\nabla(\tilde{u}_h - u_h)\|_{L^2(T)} \leq C
h^{\frac{1}{2}}_e\|(g^\prime-g_h^\prime)\|_{L^2(e)},
\end{align*}
where $C$ is a positive constant, $e\in \cE_h^b$ an edge of $T$
and  $g^\prime$ (resp. $g_h^\prime$) denotes the tangential
derivative of $g$ (resp. $g_h$) along $e$.
\end{theorem}

In subsequent discussion, we prove two lemmas before proving the
theorem. Note that $\tilde u_h \in L^2(T)$ for all $T\in\cT_h^b$.
Therefore it defines a distribution via
\begin{align*}
\tilde u_h (\phi) =\int_T \tilde u_h\phi\,dx\quad \forall
\phi\in{\mathcal D}(T),
\end{align*}
where ${\mathcal D}(T)$ is the space of all $C^\infty$ functions
having compact support in $T$. The distribution derivative
$\partial\tilde u_h /\partial x_i$ for ($i=1,2$) is given by
\cite{Kesavan:book},
\begin{align*}
\frac{\partial \tilde u_h }{\partial x_i} (\phi) =-\int_T \tilde
u_h \frac{\partial \phi}{\partial x_i}\,dx \quad \forall
\phi\in{\mathcal D}(T).
\end{align*}
Further, we have
\begin{align}
\frac{\partial (\tilde u_h-u_h)}{\partial x_i} (\phi) =-\int_T
(\tilde u_h-u_h) \frac{\partial \phi}{\partial x_i}\,dx \quad
\forall \phi\in{\mathcal D}(T).\label{eq:WeakDeri}
\end{align}
Since $u_h|_T\in P_1(T)$, we note that
\begin{align*}
u_h(x,y) &=\frac{h_2 u_h(x_0,y_0)+h_1 u_h(x_1,y_1)}{h_1+h_2}\quad
\text{ for } (x,y)\in T^\circ,
\end{align*}
and
\begin{align*}
(\tilde{u}_h - u_h)(x,y) &=
\bigg(\frac{h_2}{h_1+h_2}\bigg)(g-g_h)(x_0,y_0).
\end{align*}
From \eqref{eq:WeakDeri}, we can treat the differentiation on
$\tilde{u}_h - u_h$ under the integral (in the sense of a.e.) and
write
\begin{align}
\nabla(\tilde{u}_h - u_h)(x,y) &= \bigg(\frac{h_1\nabla
h_2-h_2\nabla h_1}{(h_1+h_2)^2}\bigg)(g-g_h)(x_0,y_0)\notag\\
&\qquad
+\bigg(\frac{h_2}{h_1+h_2}\bigg)\nabla\big((g-g_h)(x_0,y_0)\big).\label{eq:tuhderivative}
\end{align}
Using the formula for $h_1$ and $h_2$, we find
\begin{align*}
\nabla h_1
&=\bigg(\frac{(b_2-b_1)}{\sqrt{(b_1-b_2)^2+(a_1-a_2)^2}},\frac{(a_1-a_2)}{\sqrt{(b_1-b_2)^2+(a_1-a_2)^2}}\bigg),
\\\\
\nabla h_2
&=\frac{\sqrt{(b_1-b_2)^2+(a_1-a_2)^2}}{\mid(a_1-a_2)(a_i-a_3)+(b_1-b_2)(b_i-b_3)\mid}\big((b_i-b_3),(a_i-a_3)\big),
\end{align*}
and then
\begin{align*}
|\nabla h_1\cdot\nabla h_2| &=1,\\
|\nabla h_1\cdot\nabla h_1|&=1,\\
\nabla h_2\cdot\nabla h_2
&=\frac{\big({(b_1-b_2)^2+(a_1-a_2)^2})\big((b_i-b_3)^2+(a_i-a_3)^2\big)}{\mid(a_1-a_2)(a_i-a_3)+(b_1-b_2)(b_i-b_3)\mid^2}=\frac{1}{\cos^2(\theta_i)},
\end{align*}
where, $\theta_i$ is  the interior angle  of  $T$ at  $(a_i,b_i)$
for $i=1,2$. Since the triangulation is regular, the minimal angle
in the triangulation $\cT_h$ is bounded below by some angle
$\theta_0$. Therefore $\theta_i\leq \pi/2-\theta_0$ and
$\cos(\theta_i) \geq \cos(\pi/2-\theta_0)\geq C$ for some positive
constant $C$ uniformly. Further, we find by the chain rule
\begin{align*}
\frac{\partial((g-g_h)(x_0,y_0))}{\partial x} &=
\Large\frac{t_1}{h_e^2} \bigg(t_1 \frac{\partial(g-g_h)}{\partial
x_0}(x_0,y_0)+t_2\frac{\partial(g-g_h)}{\partial
y_0}(x_0,y_0)\bigg),\\
\frac{\partial((g-g_h)(x_0,y_0))}{\partial y} &=
\Large\frac{t_2}{h_e^2}\bigg(t_1\frac{\partial(g-g_h)}{\partial
x_0}(x_0,y_0)+t_2\frac{\partial(g-g_h)}{\partial
y_0}(x_0,y_0)\bigg),
\end{align*}
where $h_e =\sqrt{(a_1-a_2)^2+(b_1-b_2)^2}$, $t_1=a_1-a_2$ and
$t_2=b_1-b_2$. Let ${\bf t}=(\hat t_1,\hat t_2)=(t_1,t_2)/h_e$.
Then ${\bf t}$ is the unit tangent vector to the edge $e$ in the
direction of $(t_1,t_2)$. It is now easy to see that
\begin{align}
|\nabla(g-g_h)(x_0,y_0)|= \bigg|\hat
t_1\frac{\partial(g-g_h)}{\partial x_0}(x_0,y_0)+\hat
t_2\frac{\partial(g-g_h)}{\partial
y_0}(x_0,y_0)\bigg|.\label{eq:Gradg}
\end{align}
The term on the right hand side of the above inequality is nothing
but the tangential derivative of $(g-g_h)$ in the direction of
$(t_1,t_2)$ on the edge $e$. We now prove the following lemma:

\begin{lemma}\label{lem:Gradg}
Let $x_0(x,y)=x_0$ and $y_0(x,y)=y_0$ be defined as in \eqref{eq:x0} and \eqref{eq:y0}, respectively. Then
for $T\in\cT_h^b$ and $e=\partial T\cap\partial\Omega$, there holds
\begin{align*}
\int_T|\nabla(g-g_h)(x_0(x,y),y_0(x,y)|^2\,dxdy \leq C
h_e\|(g^\prime-g_h^\prime)\|_{L^2(e)}^2,
\end{align*}
where $g^\prime$ (resp. $g_h^\prime$) denotes the tangential
derivative of $g$ (resp. $g_h$) along $e$.
\end{lemma}

\begin{proof}
From the identity \eqref{eq:Gradg}, we find
\begin{align*}
&\int_T|\nabla(g-g_h)(x_0(x,y),y_0(x,y)|^2\,dxdy \\
&\qquad=\int_T \bigg|\hat t_1\frac{\partial(g-g_h)}{\partial
x_0}(x_0(x,y),y_0(x,y))+\hat t_2\frac{\partial(g-g_h)}{\partial
y_0}(x_0(x,y),y_0(x,y))\bigg|^2\,dxdy\\
&\qquad=\int_T |(g-g_h)^\prime(x_0(x,y),y_0(x,y))|^2\,dxdy.
\end{align*}
Now the proof follows from the observation that the integrand is constant along the lines that are orthogonal to $e$.
\end{proof}
We further prove the following lemma:
\begin{lemma}\label{lem:Gradg1}
Assume the hypothesis of Lemma $\ref{lem:Gradg}$, then there holds
\begin{align*}
\int_T\left|\bigg(\frac{h_1\nabla h_2-h_2\nabla h_1}{(h_1+h_2)^2}\bigg)(g-g_h)(x_0(x,y),y_0(x,y))\right|^2\,dxdy\leq C h_e\|(g^\prime-g_h^\prime)\|_{L^2(e)}^2.
\end{align*}
\end{lemma}

\begin{proof}
For simplicity denote $x_0(x,y)$ by $x_0$ and $y_0(x,y)=y_0$. Then
\begin{align*}
&\int_T\left|\bigg(\frac{h_1\nabla h_2-h_2\nabla h_1}{(h_1+h_2)^2}\bigg)(g-g_h)(x_0,y_0)\right|^2\,dxdy \\
&\qquad=\int_T\bigg(\frac{(\nabla h_2\cdot\nabla h_2)h_1^2+(\nabla
h_1\cdot\nabla h_1)h_2^2-2h_2h_1\nabla h_2\cdot\nabla
h_1}{(h_1+h_2)^4}\bigg)(g-g_h)^2(x_0,y_0)\,dxdy\\
&\qquad \leq\int_T\bigg(\frac{(\cos\theta_i)^{-2}h_1^2+h_2^2+2h_2h_1}{(h_1+h_2)^4}\bigg)(g-g_h)^2(x_0,y_0)\,dxdy\\
&\qquad\leq\frac{1}{(\cos\theta_i)^2}\int_T\bigg(\frac{h_1^2+h_2^2+2h_2h_1}{(h_1+h_2)^4}\bigg)(g-g_h)^2(x_0,y_0)\,dxdy\\
&\qquad \leq\frac{1}{(\cos\theta_i)^2}\int_T\frac{(g-g_h)^2(x_0,y_0)}{(h_1+h_2)^2}\,dxdy\\
&\qquad=\frac{1}{(\cos\theta_i)^2}\Bigg(\int_{T_1}\frac{(g-g_h)^2(x_0,y_0)}{(h_1+h_2)^2}dxdy+\int_{T_2}\frac{(g-g_h)^2(x_0,y_0)}{(h_1+h_2)^2}dxdy\Bigg)
\end{align*}
\par
\noindent Without loss of generality, assume that the boundary
edge is parallel to the $x$-axis and consider the integral on
$T_1$.
\begin{align*}
\int_{T_1}\frac{(g-g_h)^2(x_0,y_0)}{(h_1+h_2)^2}dxdy
&=\int_{a_1}^{a_3}\int_{b_1}^{b_1+\big(\frac{b_3-b_1}{a_3-a_1}\big)(x-a_1)}\frac{(g-g_h)^2(x_0,y_0)}{\Big(\big(\frac{b_3-b_1}{a_3-a_1}\big)(x-a_1)\Big)^2}\,dxdy\\
&=\bigg(\frac{a_3-a_1}{b_3-b_1}\bigg)\int_{a_1}^{a_3}\frac{(g-g_h)^2(x,b_1)}{(x-a_1)}dx.
\end{align*}
Since $g\in H^1(\bar e)$, $g$ is absolutely continuous on $e$
\cite[Theorem 2.1.2]{Kesavan:book}. Then as $(g-g_h)(a_1,b_1)=0$,
we find using a property of absolutely continuous functions
\cite[Theorem 14 on page 110]{Royden:book} that for any $x\in
(a_1,a_3)$,
\begin{align*}
(g-g_h)(x,b_1)=\int_{a_1}^x(g-g_h)^\prime(t,b_1)\,dt.
\end{align*}
Subsequently, there holds
\begin{align*}
|(g-g_h)(x,b_1)|^2\leq |x-a_1|\;\|(g-g_h)^\prime\|_{L_2(e)}^2.
\end{align*}
Also since $(a_3-a_1)/(b_3-b_1)\leq C$ for some constant $C$ which
is independent of $T$, we obtain
\begin{align*}
\int_{T_1}\frac{(g-g_h)^2(x_0,y_0)}{(h_1+h_2)^2}dxdy \leq
\bigg(\frac{a_3-a_1}{b_3-b_1}\bigg)(a_3-a_1)\,\|(g-g_h)^\prime\|_{L_2(e)}^2
\leq C h_e \|(g-g_h)^\prime\|_{L_2(e)}^2.
\end{align*}
Using the same arguments, we obtain the similar estimate for $T_2$. Finally using the fact that $cos(\theta_i)\geq C$ for some positive constant, we conclude the proof.
\end{proof}

{\em Proof of } {\bf Theorem \ref{thm:Oscg}}: Using the equation
\eqref{eq:tuhderivative} and since
\begin{align*}
\bigg|\frac{h_2}{h_1+h_2}\bigg| \leq 1,
\end{align*}
we complete the proof by using the triangle inequality, Lemma
\ref{lem:Gradg} and Lemma \ref{lem:Gradg1}.

\section{A posteriori Error Analysis}\label{sec:ErrorAnalysis}
\par
\noindent In this section, we derive a reliable and efficient a
posteriori error estimator. The error analysis can conveniently be
derived by using an appropriate residual. To this end, define the
residual $\cR_h:H^1_0(\O)\rightarrow \R$ by
\begin{align}
\cR_h(v):=(f,v)-a(\tilde u_h, v)-(\sigma_h,v)\label{eq:Residual1}
\end{align}
or equivalently, define
\begin{align}
\cR_h(v):=a(u-\tilde u_h, v)+\langle
\sigma-\sigma_h,v\rangle,\label{eq:Residual}
\end{align}
where $\tilde u_h$ is either the harmonic extension or the linear
extension as in Section \ref{sec:post-solu}.

\par
\noindent The following lemma establishes a relation between the
residual and the errors.

\begin{lemma}\label{lem:residual-error}
There holds
\begin{align*}
\|\nabla(u-u_h)\|^2+\|\sigma-\sigma_h\|_{-1}^2\leq 6\,
\|\cR_h\|_{-1}^2-8 \,\langle \sigma-\sigma_h,u-\tilde u_h\rangle
+2\, \|\nabla(u_h-\tilde u_h)\|^2.
\end{align*}
\end{lemma}
\begin{proof}
From \eqref{eq:Residual} and Young's inequality, we find
\begin{align*}
\|\nabla(u-\tilde u_h)\|^2 &= a(u-\tilde u_h, u-\tilde
u_h)=\cR_h(u-\tilde u_h)-\langle \sigma-\sigma_h, u-\tilde u_h\rangle\\
& \leq \|\cR_h\|_{-1}\;\|\nabla(u-\tilde u_h)\|-\langle
\sigma-\sigma_h,u-\tilde u_h\rangle\\
&\leq \frac{1}{2}\|\cR_h\|_{-1}^2+\frac{1}{2}\|\nabla(u-\tilde
u_h)\|^2-\langle \sigma-\sigma_h,u-\tilde u_h\rangle,
\end{align*}
which proves
\begin{align}\label{eq:resi1-error}
\|\nabla(u-\tilde u_h)\|^2 &\leq \|\cR_h\|_{-1}^2-2\langle
\sigma-\sigma_h,u-\tilde u_h\rangle.
\end{align}
Then the triangle inequality and Young's inequality imply
\begin{align}
\|\nabla(u-u_h)\|^2 &\leq 2 \|\nabla(u-\tilde u_h)\|^2+
2\|\nabla(u_h-\tilde u_h)\|^2 \notag\\
&\leq 2\|\cR_h\|_{-1}^2-4\langle \sigma-\sigma_h,u-\tilde
u_h\rangle+ 2\|\nabla(u_h-\tilde u_h)\|^2.\label{eq:resi2-error}
\end{align}
Again from \eqref{eq:Residual}, we find
\begin{align*}
\|\sigma-\sigma_h\|_{-1}\leq \|\cR_h\|_{-1}+ \|\nabla(u-\tilde
u_h)\|,
\end{align*}
then by Young's inequality and \eqref{eq:resi1-error},
\begin{align}
\|\sigma-\sigma_h\|_{-1}^2 &\leq 2\|\cR_h\|_{-1}^2 + 2
\|\nabla(u-\tilde u_h)\|^2 \notag\\
&\leq 4\|\cR_h\|_{-1}^2-4\langle \sigma-\sigma_h,u-\tilde
u_h\rangle. \label{eq:resi3-error}
\end{align}
 The proof then follows by combining the estimates in \eqref{eq:resi2-error}-\eqref{eq:resi3-error}.
\end{proof}

In order to derive the error analysis, we define the following
error estimators :
\begin{align*}
\eta_{f,T} &:=h_T\|f-\sigma_h\|_{L_2(T)},\\
\eta_{u_h,e} &:=h_e^{1/2}\|\sjump{\nabla u_h}\|_{L_2(e)},\\
\eta_{\sigma,T}&:=h_T^2\|\nabla \sigma_h\|_{L_2(T)},
\end{align*}
and
\begin{align*}
\eta_{f}^2 :=\sum_{T\in\cT_h}\eta_{f,T}^2,\qquad \eta_{J}^2
&:=\sum_{e\in\cE_h^i}\eta_{u_h,e}^2,\qquad
\eta_{\sigma}^2:=\sum_{T\in\cT_h}\eta_{\sigma,T}^2,\qquad
\eta_g^2:=\|\nabla(u_h-\tilde u_h)\|^2,
\end{align*}
and finally define
\begin{align*}
\eta^2&:=\eta_f^2+\eta_J^2+\eta_\sigma^2+\eta_g^2.
\end{align*}

\par
\noindent In the following lemma, we estimate the dual norm of the
residual $\cR_h$.

\begin{lemma}\label{lem:Rh}
There holds
\begin{align*}
\|\cR_h\|_{-1}\leq C\eta.
\end{align*}
\end{lemma}

\begin{proof}
From \eqref{eq:Residual1} and \eqref{eq:sigmah-def}, we note that
\begin{align*}
\cR_h(v)&=(f,v-v_h)-a(u_h, v-v_h)-(\sigma_h,v-v_h)+\langle
\sigma_h,v_h \rangle_h-(\sigma_h,v_h)+ a(u_h-\tilde u_h, v),
\end{align*}
for any $v_h\in V_h$. Let $v_h$ be the Clement interpolation of
$v$. Then the proof now follows by integration by parts, the
approximation properties of $v_h$ and the estimate in
\eqref{eq:NumerQuad}.
\end{proof}

\par
\noindent It is now remaining to estimate $-\langle
\sigma-\sigma_h,u-\tilde u_h\rangle$. The error estimator depends
on the following sets defined by
\begin{align*}
\bF_h :=\{ T\in\cT_h: \exists\; z_1,\;z_2\in\cV_T \text{ such that
} u_h(z_1)=\chi(z_1) \text{ and } u_h(z_2)>\chi(z_2)\},
\end{align*}
\begin{align*}
\bC_h:=\{ T\in\cT_h: \text{For all } z\in\cV_T,\;
u_h(z)=\chi(z)\},
\end{align*}
and
\begin{align*}
\bN_h :=\{ T\in\cT_h: \text{For all } z\in\cV_T,\;
u_h(z)>\chi(z)\}.
\end{align*}
We call the sets $\bF_h$, $\bC_h$ and $\bN_h$ as free boundary
set, contact set and non-contact set, respectively. We also define
\begin{align*}
\p\bF_h^i :=\{e\in \cE_h^i: e\in \cE_z, z\in \cV_T \text{  with  }
u_h|_T(z)=\chi_h(z),\; T \in \bF_h\},
\end{align*}
the set of all interior edges that share a vertex of a triangle in
$\bF_h$.

\par
\noindent We prove the following lemma in the spirit of
\cite{Veeser:2001:VI,TG:2014:VIDG}:

\begin{lemma}\label{lem:Lower}
Let $\chi_h\in V_h$ be the Lagrange interpolation of $\chi$. Then
there holds
\begin{align*}
\langle \sigma-\sigma_h, u-\tilde u_h\rangle &\geq  -\epsilon
\|\sigma-\sigma_h\|_{-1}^2-\epsilon\|\nabla(u_h-
u)\|^2-\frac{C}{\epsilon}\|\nabla(\chi-\tilde
u_h)^+\|^2\\&-C\|\nabla(u_h- \tilde u_h)\|^2 -\frac{C}{\epsilon}
\sum_{T\in\cT_h} h_T^4 \|\nabla
\sigma_h(u_h)\|^2+C \sum_{T\in\bF_h\cup\bC_h}\int_T\sigma_h(\chi-\chi_h)^-\,dx\\
&-C\sum_{e\in\p\bF_h^i}\int_e h_e \sjump{\nabla
(u_h-\chi_h)}^2\,ds.
\end{align*}
\end{lemma}
\begin{proof}
Let $\hat u_h =\max\{\tilde u_h,\chi\}$. Then $\hat u_h\in \cK$
and find
\begin{align*}
\langle \sigma, u-\tilde u_h\rangle &=\langle \sigma,(u-\hat
u_h)+\hat u_h-\tilde u_h)\rangle\\
&\geq \langle \sigma,\hat u_h-\tilde u_h\rangle,
\end{align*}
where we have used the fact that  $\langle \sigma,u-\hat
u_h\rangle \geq 0$, thanks to \eqref{eq:sigma}. Note by the
definition of $\hat u_h$ that $\hat u_h-\tilde u_h=(\chi-\tilde
u_h)^+$. Then
\begin{align*}
\langle \sigma,\hat u_h-\tilde u_h\rangle & = \langle
\sigma,(\chi-\tilde u_h)^+\rangle = \langle
\sigma-\sigma_h,(\chi-\tilde u_h)^+\rangle+\langle
\sigma_h,(\chi-\tilde u_h)^+\rangle\\
&\geq -\epsilon
\|\sigma-\sigma_h\|_{-1}^2-\frac{1}{\epsilon}\|\nabla(\chi-\tilde
u_h)^+\|^2+\langle \sigma_h,(\chi-\tilde u_h)^+\rangle,
\end{align*}
where $\epsilon >0$ is some constant which will be chosen later.
Therefore
\begin{align*}
\langle \sigma-\sigma_h, u-\tilde u_h\rangle &\geq -\epsilon
\|\sigma-\sigma_h\|_{-1}^2-\frac{1}{\epsilon}\|\nabla(\chi-\tilde
u_h)^+\|^2+\langle \sigma_h,(\chi-\tilde u_h)^+-(u-\tilde
u_h)\rangle.
\end{align*}
We consider
\begin{align*}
\langle \sigma_h,(\chi-\tilde u_h)^+-(u-\tilde u_h)\rangle
&=\sum_{T\in\cT_h}\int_T \sigma_h(\chi-\tilde u_h)^+-(u-\tilde
u_h)\,dx\\
&=\sum_{T\in\bF_h\cup\bC_h}\int_T \sigma_h(\chi-\tilde
u_h)^+-(u-\tilde u_h)\,dx,
\end{align*}
since $\sigma_h\equiv 0$ on $\bN_h$. We split the proof into a few
cases.

\par
\noindent {\em Case 1.}({\em Free boundary set near $\p\O$},
$\bF_h$): Let $ T\in \bF_h$ and $T\cap \p\O\neq \emptyset$. Then
since there is at least one node $z\in \cV_T$ such that
$\sigma_h|_T(z)=0$, we find by using equivalence of norms on
finite dimensional spaces and scaling that
\begin{align}
\|\sigma_h\|_{L_2(T)} &\leq C h_T
\|\nabla\sigma_h\|_{L_2(T)}.\label{eq:SigmaPoincare}
\end{align}
If $T\cap \p\O$ is an edge in $\cE_h^b$, then since $u-\tilde u_h
\in H^1_0(\Omega)$, we find by using Poincar\'e inequality and
scaling that
\begin{align}
\|\tilde u_h- u\|_{L_2(T)} &\leq C h_T \|\nabla(\tilde u_h-
u)\|_{L_2(T)}.\label{eq:FBB1}
\end{align}
In the case if $T\cap \p\O$ is a node in $\cV_h^b$, we find by
using the equivalence of norms on finite dimensional spaces  that
\begin{align*}
\|\Pi_h(\tilde u_h- u)\|_{L_2(T)} &\leq C h_T \|\nabla\Pi_h(\tilde
u_h- u)\|_{L_2(T)}\leq C h_T \|\nabla(\tilde u_h-
u)\|_{L_2(\cT_T)},
\end{align*}
where $\cT_T:=\cup\{T\in \cT_h: T'\cap T\neq\emptyset\}$ and
$\Pi_h: H^1_0(\Omega)\rightarrow V_h^0$ is the Clement
interpolation satisfying
\begin{align*}
h_T^{-1}\|\|\Pi_hv-v\|_{L_2(T)}+\|\nabla\Pi_hv\|_{L_2(T)}\leq C
\|\nabla v\|_{L_2(\cT_T)}\quad \forall v\in H^1_0(\Omega).
\end{align*}
Then, we find
\begin{align}
\|\tilde u_h-u\|_{L_2(T)} &\leq \|(\tilde u_h- u)-\Pi_h(\tilde
u_h- u)\|_{L_2(T)}+\|\Pi_h(\tilde u_h- u)\|_{L_2(T)}\notag\\
&\leq C h_T \|\nabla(\tilde u_h- u)\|_{L_2(\cT_T)}.\label{eq:FBB3}
\end{align}
\noindent Therefore, we obtain
\begin{align}
\int_T \sigma_h(\tilde u_h-u)\,dx &\geq
-\frac{\epsilon}{2}\|\nabla(\tilde u_h-
u)\|_{L_2(\cT_T)}^2-\frac{C}{\epsilon} h_T^4 \|\nabla
\sigma_h(u_h)\|_{L_2(T)}^2\notag\\
&\geq -\epsilon\|\nabla(\tilde
u_h-u_h)\|_{L_2(\cT_T)}^2-\epsilon\|\nabla(u_h-
u)\|_{L_2(\cT_T)}^2\notag\\& \qquad-C h_T^4 \|\nabla
\sigma_h(u_h)\|_{L_2(T)}^2. \label{eq:Lower2}
\end{align}
Similarly, we find that
\begin{align}
\int_T \sigma_h(\chi-\tilde u_h)^+\,dx &\geq -
\|\nabla(\chi-\tilde u_h)^+\|_{L_2(\cT_T)}^2-C h_T^4 \|\nabla
\sigma_h(u_h)\|_{L_2(T)}^2. \label{eq:Lower22}
\end{align}

\bigskip \noindent For the remaining proof, we note that
\begin{align*}
\int_T \sigma_h\big((\chi-\tilde u_h)^+-(u-\tilde u_h)\big)\,dx
&=\int_T
\sigma_h\big((\chi-\tilde u_h)^+-(\chi-\tilde u_h)\big)\,dx\notag\\
&\quad + \int_T \sigma_h(\chi-u)\,dx\notag\\
&\geq \int_T \sigma_h\big((\chi-\tilde u_h)^+-(\chi-\tilde
u_h)\big)\,dx,
\end{align*}
since $\sigma_h\leq 0$ and $\chi\leq   u$. Further since any
$\phi=\phi^+ - \phi^-$, where $\phi^-= \max\{-\phi,0\}$, we note
that
\begin{align}
\int_T \sigma_h \big((\chi-\tilde u_h)^+ -(u-\tilde u_h)\big)\,dx
= \int_T \sigma_h (\chi-\tilde u_h)^- \,dx.\label{eq:sigma-ine}
\end{align}

\par
\smallskip
{\em Case 2}({\em Free boundary set away from $\p\O$}, $\bF_h$):
Let $T\in \bF_h$ and $T\cap \p\O=\emptyset$. Then $\tilde
u_h\equiv u_h$ on $T$ and
\begin{align*}
\int_T\sigma_h (\chi-\tilde u_h)^- \,dx &= \int_T\sigma_h
(\chi-u_h)^- \,dx\\
&\geq \int_T\sigma_h (\chi-\chi_h)^- \,dx+ \int_T\sigma_h
(u_h-\chi_h) \,dx.
\end{align*}

\par
\noindent Since $T\in \bF_h$, there is some $z\in \cV_T$ such that
$u_h|_T(z)=\chi(z)$. Also since $ u_h\geq \chi_h$, we find by
using \cite[Lemma 3.6]{Veeser:2001:VI} that
\begin{align}
\|u_h-\chi_h\|_{L_2(T)}\leq C h_T \left(\sum_{e\in\cE_p\cap
\cE_h^i}\int_e h_e \sjump{\nabla (
u_h-\chi_h)}^2\,ds\right)^{1/2}.\label{eq:Free2}
\end{align}
Also since $\sigma_h|_T(z)=0$ for at least one $z\in\cV_T$, we
find using \eqref{eq:SigmaPoincare} and \eqref{eq:Free2} that
\begin{align*}
\big|\int_T \sigma_h (u_h-\chi_h)\,dx\big|
&\leq  \|\sigma_h\|_{L_2(T)} \|  u_h-\chi_h \|_{L_2(T)} \notag\\
&\leq C h_T^2 \|\nabla
\sigma_h\|_{L_2(T)}\left(\sum_{e\in\cE_p\cap \cE_h^i}\int_e h_e
\sjump{\nabla (u_h-\chi_h)}^2\,ds\right)^{1/2},
\end{align*}
which implies that
\begin{align}
\int_T \sigma_h( u_h-\chi_h)\,dx &\geq -C h_T^4 \|\nabla
\sigma_h\|_{L_2(T)}^2 -C\sum_{e\in\cE_p\cap \cE_h^i}\int_e h_e
\sjump{\nabla (u_h-\chi_h)}^2\,ds. \label{eq:Lower3}
\end{align}

\par
\smallskip
{\em Case 3}({\em Contact set away from $\p\O$}, $\bC_h$): Let
$T\in \bC_h$ and $T\cap \p\O=\emptyset$. Then $\tilde u_h\equiv
u_h\equiv \chi_h$ on $T$ and
\begin{align*}
\int_T\sigma_h (\chi-\tilde u_h)^- \,dx = \int_T\sigma_h
(\chi-u_h)^- \,dx=\int_T\sigma_h (\chi-\chi_h)^- \,dx.
\end{align*}

\par
\smallskip
{\em Case 4}({\em Contact set near $\p\O$}, $\bC_h$): Let $T\in
\bC_h$ and $T\cap \p\O\neq \emptyset$. Then $u_h\equiv\chi_h$ on
$T$ and
\begin{align*}
(\chi-\tilde u_h)^- &=\max\{\tilde u_h-\chi,0\}=\max\{\tilde
u_h-u_h+\chi_h-\chi,0\}\\
&\leq \max\{\tilde u_h-u_h,0\}+\max\{\chi_h-\chi,0\}\\
&=(\tilde u_h-u_h)^+ +(\chi-\chi_h)^-.
\end{align*}
Therefore
\begin{align*}
\int_T\sigma_h (\chi-\tilde u_h)^- \,dx \geq  \int_T\sigma_h
(\tilde u_h-u_h)^+ \,dx+\int_T\sigma_h (\chi-\chi_h)^- \,dx.
\end{align*}
Further since $\sigma_h(z)=0$ for $z\in\p T\cap \p\O$,
we find by using equivalence of norms on
finite dimensional spaces and scaling as in \eqref{eq:SigmaPoincare} that
\begin{align*}
\|\sigma_h\|_{L_2(T)} &\leq C h_T
\|\nabla\sigma_h\|_{L_2(T)}.
\end{align*}
and similar to \eqref{eq:FBB1}-\eqref{eq:FBB3}, we find that
\begin{align*}
\|(\tilde u_h-u_h)^+\|_{L_2(T)} \leq \|\tilde u_h-u_h\|_{L_2(T)}
&\leq C h_T \|\nabla(\tilde u_h- u_h)\|_{L_2(\cT_T)}.
\end{align*}
This completes the rest of the proof.
\end{proof}
We combine Lemma \ref{lem:residual-error}, Lemma \ref{lem:Rh},
Lemma \ref{lem:Lower} and deduce the following main result of the
article.

\begin{theorem}\label{thm:Reliable}
Let $\chi_h\in V_h$ be the Lagrange interpolation of $\chi$ and
$g|_e\in H^1(e)$ for all $e\in\cE_h^b$. Then  there holds
\begin{align*}
\|\nabla(u-u_h)\|^2 +\|\sigma-\sigma_h\|_{-1}^2 &\leq C (\eta_f^2+\eta_J^2+\eta_\sigma^2)+C \sum_{e\in \cE_h^b} h_e\|(g-g_h)^\prime\|_{L_2(e)}^2\\
&\quad+C\|\nabla(\chi-\tilde u_h)^+\|^2+C \sum_{T\in\bF_h\cup\bC_h}\int_T(-\sigma_h)(\chi-\chi_h)^-\,dx\\
&\quad +C\sum_{e\in\p\bF_h^i}\int_e h_e \sjump{\nabla
(u_h-\chi_h)}^2\,ds.
\end{align*}
\end{theorem}

\begin{remark}
The error estimator in general  explicitly depends on the
post-processed solution $\tilde u_h$. If $\tilde u_h$ is defined
by the harmonic extension, then the error estimator is not
computable. However, the linear extension can be used to construct
$\tilde u_h$ and the error estimator is computable.
\end{remark}

Under an assumption that the free boundary is interior to $\Omega$, we prove the following estimate:
Note that the hypothesis we make in the corollary is easy to check.

\begin{corollary}\label{cor:Reliable}
Suppose that for all $T\in \cT_h^b$, there holds
\begin{align*}
\max_{z\in T}\chi(z)\leq \min_{z\in \partial T} u_h^*,
\end{align*}
where $u_h^*$ is defined as in \eqref{eq:uht}. Then, we have
\begin{align*}
\|\nabla(u-u_h)\|^2 +\|\sigma-\sigma_h\|_{-1}^2 &\leq C (\eta_f^2+\eta_J^2+\eta_\sigma^2)+C \sum_{e\in \cE_h^b} h_e\|(g-g_h)^\prime\|_{L_2(e)}^2\\
&\quad+C \sum_{T\in\cT_h^i} \|\nabla(\chi-
u_h)^+\|_{L_2(T)}^2+C \sum_{T\in\bF_h\cup\bC_h}\int_T(-\sigma_h)(\chi-\chi_h)^-\,dx\\
&\quad+C\sum_{e\in\p\bF_h^i}\int_e h_e \sjump{\nabla
(u_h-\chi_h)}^2\,ds.
\end{align*}
In particular, the error estimator is independent of the
post-processed solution $\tilde u_h$.
\end{corollary}
\begin{proof}
From \eqref{eq:minP} and \eqref{eq:minP1}, we have for any
$T\in\cT_h^b$ that
\begin{align*}
\min_{z\in T}\tilde u_h(z)=\min_{z\in \partial T} u_h^*.
\end{align*}
Hence by the hypothesis in the statement of the corollary,
$(\chi-\tilde u_h)^+\equiv 0$ on any $T\in\cT_h^b$. Further since
$\tilde u_h\equiv u_h$ on any $T\in\cT_h^i$, we have on any such
$T$ that $(\chi-\tilde u_h)^+\equiv (\chi-u_h)^+$. Moreover, from
\eqref{eq:OscHg} and Theorem \ref{thm:Oscg}, we have
\begin{align*}
\eta_g^2=\|\nabla(u-\tilde u_h)\|^2\leq C \sum_{e\in \cE_h^b} h_e\|(g-g_h)^\prime\|_{L_2(e)}^2.
\end{align*}
The proof now follows from Theorem \ref{thm:Reliable}.
\end{proof}

\begin{corollary}\label{cor:P1obstacle}
Suppose that $\chi\in \bbP_1(\Omega)$. Then
\begin{align*}
\|\nabla(u-u_h)\|^2 +\|\sigma-\sigma_h\|_{-1}^2 &\leq C
(\eta_f^2+\eta_J^2+\eta_\sigma^2)+C \sum_{e\in \cE_h^b}
h_e\|(g-g_h)\,^\prime\|_{L_2(e)}^2.
\end{align*}
\end{corollary}
\begin{proof}
Since $\chi\in \bbP_1(\Omega)$, we have that $u_h\geq \chi$ on
$\bar\Omega$. From the hypothesis on $g$, we have $g\geq \chi$ on
$\partial\Omega$. These together imply that $u_h^*\geq \chi$ on
$\partial T$ for any $T\in\cT_h^b$. Further by the construction of
$\tilde u_h$ (in both the choices), we have for any $T\in \cT_h^b$
that
\begin{align*}
\max_{z\in T}\tilde u_h=\max_{z\in \partial T}u_h^*.
\end{align*}
Therefore $(\chi-\tilde u_h)^+\equiv 0$ on $\bar\Omega$. This
together with Theorem \ref{thm:Reliable} completes the proof.
\end{proof}

\begin{remark}
The post-processed solution $\tilde u_h$ defined by the linear
extension is computable once the discrete solution $u_h$ is known.
The new solution $\tilde u_h$ satisfies the exact boundary
conditions. It may also be useful in other contexts.
\end{remark}

\section{Simplified Error Estimator}\label{sec:Simplified}
In this section, we conclude a simplified error estimator for the
obstacle problem by equivalently rewriting the model problem
\eqref{eq:MP} as follows: Find $u\in \cK$ such that $u=w+\chi$,
where $w\in\cK_0$ satisfies
\begin{align}\label{eq:RMP}
a(w,v-w)\geq (f,v-w)-a(\chi,v-w)\;\; \text{ for all } v\in \cK_0,
\end{align}
and
\begin{align*}
\cK_0:=\{v\in H^1(\Omega): v\geq 0 \text{ a.e. in } \Omega,\;\;
\gamma_0(v)=g-\chi \text{ on } \partial\Omega \},
\end{align*}
It can be seen that the set $\cK_0$
is nonempty as $u-\chi \in \cK_0$. The discrete version of $\cK_0$ can be defined as
\begin{align*}
\cK_{0h}:=\{v\in V_h: v\geq 0 \text{ a.e. in } \Omega,\;\;
\gamma_0(v)=g_h-\chi_h \text{ on } \partial\Omega \},
\end{align*}
where $\chi_h\in V_h$ is the Lagrange interpolation of $\chi$. The
discrete problem \eqref{eq:DP} also can be equivalently rewritten
as to find $u_h\in \cK_h$ such that $u_h=w_h+\chi_h$, where
$w_h\in\cK_{0h}$ satisfies
\begin{align}\label{eq:RDP}
a(w,v-w)\geq (f,v-w)-a(\chi_h,v-w)\;\; \text{ for all } v\in \cK_{0h}.
\end{align}

\par
\noindent Let $\gamma_0(\chi)\in H^1(e)$ for all $e\in \cE_h^b$.
Then the Dirichlet data $g-\chi$ satisfies the hypothesis of
Theorem \ref{thm:Reliable} and hence the  Corollary
\ref{cor:P1obstacle}  deduce the following result under the
assumption that $\gamma_0(\chi)\in H^1(e)$ for all $e\in \cE_h^b$:

\begin{theorem}\label{thm:zeroobstacle}
Let $u\in\cK$  $(resp.\;\; u_h \in\cK_h)$ be the solution of
$\eqref{eq:MP}$ $(resp.\; \eqref{eq:DP})$. Assume that the
obstacle $\chi$ and the Dirichlet data $g$ satisfy
$\gamma_0(\chi)\in H^1(e)$ and $g\in H^1(e)$ for all $e\in
\cE_h^b$. Then there holds
\begin{align*}
\|\nabla(u-u_h)\|^2 +\|\sigma-\sigma_h\|_{-1}^2 &\leq C
(\eta_f^2+\eta_J^2+\eta_\sigma^2)+C \|\nabla(\chi-\chi_h)\|^2\\
&\quad+C \sum_{e\in \cE_h^b} h_e\|(g-g_h)\,^\prime\|_{L_2(e)}^2+ C
\sum_{e\in \cE_h^b} h_e\|(\chi-\chi_h)\,^\prime\|_{L_2(e)}^2.
\end{align*}
\end{theorem}

\par
\noindent Note that the error estimator in Theorem
\ref{thm:zeroobstacle} does not involve any min/max functions
(which are non smooth and not easy to compute in the
implementation).

\subsection{Local Efficiency Estimates:}
The terms
\begin{align*}
\left(\sum_{e\in \cE_h^b}
h_e\|(g-g_h)\,^\prime\|_{L_2(e)}^2\right)^{1/2},\quad
\left(\sum_{e\in \cE_h^b}
h_e\|(\chi-\chi_h)\,^\prime\|_{L_2(e)}^2\right)^{1/2}
\quad\text{and}\quad \|\nabla(\chi-\chi_h)\|
\end{align*}
are considered to be the data approximations. We will prove the
local efficiency of the remaining error estimators.

\begin{theorem}\label{thm:Eff}
There holds
\begin{align}
 h_T \|f-\sigma_h\|_{L_2(T)}&\leq C\big(|u-u_h|_{H^1(T)}+\|\sigma-\sigma_h\|_{H^{-1}(T)}+h_T\|f-\bar
 f\|_{L_2(T)}\big),\label{eq:oscf2}\\
h_T^2 \|\nabla \sigma_h\|_{L_2(T)}&\leq
C\big(|u-u_h|_{H^1(T)}+\|\sigma-\sigma_h\|_{H^{-1}(T)}+h_T\|f-\bar
 f\|_{L_2(T)}\big),\label{eq:effsigma}
\end{align}
for any $T\in \cT_h$ and $\bar f\in \bbP_0(T)$. Furthermore for
any $e\in\cE_h^i$, there holds
\begin{align}
 h_e^{1/2}\|\sjump{\nabla u_h}\|_{L_2(e)}&\leq C\big(|u-u_h|_{H^1(\cT_e)}+\|\sigma-\sigma_h\|_{H^{-1}(\cT_e)}+h_e\|f-\bar
 f\|_{L_2(\cT_e)}\big),\label{eq:effjump}
\end{align}
where $\cT_e$ is the patch of two triangles $T_\pm$ sharing the
edge $e$, see Fig \ref{Fig1}.
\end{theorem}

\begin{proof}
The proof follows by bubble function techniques
\cite{Verfurth:1994:Adaptive,Verfurth:1995:AdaptiveBook}
\cite{Veeser:2001:VI}. Let $\bar f|_T \in \bbP_{0}(T)$ for all
$T\in\cT_h$.
 Let $T\in\cT_h$ be arbitrary and $b_T\in \bbP_3(T)\cap H^1_0(T)$ be the bubble function defined on $T$ such that
 $b_T(x_T)=1$, where $x_T$ is the barycenter of $T$.
 Let $\phi=b_T(\bar f-\sigma_h)$ on $T$ and extend it to be zero on $\Omega\backslash T$.
 We have for some mesh independent constants $C_1$ and $C_2$ that
 $$C_1 \|\bar f-\sigma_h\|_{L_2(T)} \leq \|\phi\|_{L_2(T)}\leq C_2 \|\bar f-\sigma_h\|_{L_2(T)}.$$
 It follows from \eqref{eq:sigmadef} and integration by parts that
\begin{align*}
 \int_T (f-\sigma_h)\phi\,dx &=\int_T (f-\sigma_h) \phi\,dx+\int_T \Delta u_h\, \phi\,dx \\
&=\int_\Omega f \phi\,dx-\int_\Omega \nabla u_h\cdot\nabla
\phi\,dx-\langle \sigma_h,\phi\rangle\\
&=\int_\Omega \nabla (u-u_h) \cdot\nabla \phi\,dx+\langle
\sigma-\sigma_h,\phi\rangle\\ &=\int_T \nabla (u-u_h) \cdot\nabla
\phi\,dx+\langle \sigma-\sigma_h,\phi\rangle_{H^{-1}(T)\times
H^1_0(T)}
 \end{align*}
 Using a standard inverse estimate \cite{BScott:2008:FEM,Ciarlet:1978:FEM},   we find
\begin{align*}
 C_1 \|\bar f-\sigma_h\|_{L_2(T)}^2&\leq \int_T (\bar f-\sigma_h) \phi\,dx
     =\int_T(\bar f-f)\phi\,dx+\int_T (f-\sigma_h)\phi\,dx\\
     &=\int_T(\bar f-f)\phi\,dx+\int_T \nabla (u-u_h)
\cdot\nabla \phi\,dx+\langle
\sigma-\sigma_h,\phi\rangle_{H^{-1}(T)\times H^1_0(T)}\\
     &\leq \|f-\bar f\|_{L_2(T)}\|\phi\|_{L_2(T)}
                 +\left(|u-u_h|_{H^1(T)}+\|\sigma-\sigma_h\|_{H^{-1}(T)}\right)|\phi|_{H^1(T)}\\
     &\leq \Big(\|f-\bar f\|_{L_2(T)}+C h_T^{-1}\left(|u-u_h|_{H^1(T)}+\|\sigma-\sigma_h\|_{H^{-1}(T)}\right)\Big)\|\phi\|_{L_2(T)}\\
     &\leq C\Big(\|f-\bar f\|_{L_2(T)}+h_T^{-1}\left(|u-u_h|_{H^1(T)}+\|\sigma-\sigma_h\|_{H^{-1}(T)}\right)\Big)\|\bar f-\sigma_h\|_{L_2(T)},
\end{align*}
which implies
\begin{equation}\label{eq:oscf1}
 h_T \|\bar f-\sigma_h\|_{L_2(T)}\leq C\big(|u-u_h|_{H^1(T)}+\|\sigma-\sigma_h\|_{H^{-1}(T)}+h_T\|f-\bar f\|_{L_2(T)}\big).
\end{equation}
Using the triangle inequality, we obtain
\begin{equation*}
 h_T \|f-\sigma_h\|_{L_2(T)}\leq C\big(|u-u_h|_{H^1(T)}+\|\sigma-\sigma_h\|_{H^{-1}(T)}+h_T\|f-\bar f\|_{L_2(T)}\big).
\end{equation*}
This completes the proof of \eqref{eq:oscf2}. Using an inverse
inequality \cite{BScott:2008:FEM,Ciarlet:1978:FEM}, note that
\begin{equation*}
 h_T^2 \|\nabla \sigma_h\|_{L_2(T)}=h_T^2 \|\nabla (\bar f-\sigma_h)\|_{L_2(T)}\leq C h_T\|\bar
 f-\sigma_h\|_{L_2(T)}.
\end{equation*}
An appeal to \eqref{eq:oscf1}, completes the proof of
\eqref{eq:effsigma}. We next prove \eqref{eq:effjump} using
similar bubble function techniques.
 Let $e\in \cE_h^i$ and $T_\pm$ be the triangles sharing this edge $e$. Denote by $\cT_e$ the patch of the two triangles $T_\pm$ (cf. Fig.~\ref{Fig1}).

 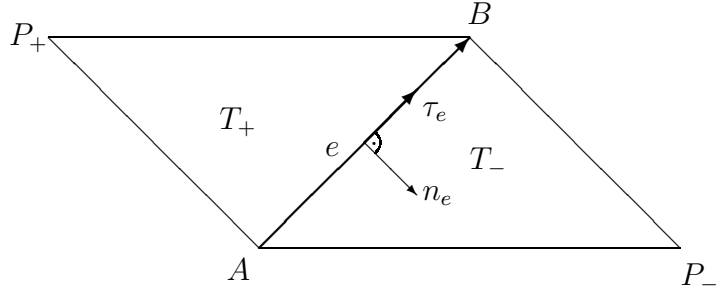
\begin{figure}[!hh]
\begin{center}
\setlength{\unitlength}{0.7cm}
\begin{picture}(8,6)
\put(2,1){\line(1,0){8}} \put(2,1){\line(-1,1){4}}
\put(6,5){\line(-1,0){8}} \put(10,1){\line(-1,1){4}}
\thicklines\put(2,1){\vector(1,1){4}}\thinlines
\put(1.39,0.39){$A$} \put(5.95,5.25){$B$} \put(10,0.29){$P_{-}$}
\put(-2.75,4.85){$P_{+}$} \put(1.25,3.25){$T_{+}$}
\put(6,2.5){$T_{-}$} \put(4,3){\vector(1,-1){1}}
\put(5.10,1.90){$n_{e}$}
\thicklines\put(4,3){\vector(1,1){1}}\thinlines
\put(5.10,3.50){$\tau_{e}$} \put(3.25,2.75){$e$}
\qbezier(4.2,2.8)(4.4,3.0)(4.2,3.2) \put(4.1,2.95){.}
\end{picture}
\caption{Two neighboring triangles $T_+$ and $T_{1}$ that share
the edge $e=\partial T_+\cap\partial T_-$ with initial node $A$
and end node $B$ and unit normal $n_e$. The orientation of $n_e =
n_{+} = -n_{-}$ equals the outer normal of $T_{+}$, and hence,
points into $T_{-}$.} \label{Fig1}
\end{center}
\end{figure}

 Now consider $\sjump{\nabla u_h}$ on $e$ and extend it to $T_\pm$ so that it is constant along the lines orthogonal to $e$.
 Denote the resulting function by $\zeta_1\in \bbP_{0}(\cT_e)$. It is then obvious that $\zeta_1=\sjump{\nabla u_h}$ on $e$.
Construct a piecewise polynomial bubble function $\zeta_2\in
H_0^1(\cT_e)$ such that $\zeta_2(x_e)=1$, where $x_e$ is the
midpoint of $e$. Denote $\phi=\zeta_1\zeta_2$ and extend $\phi$ to
be zero on $\Omega\backslash\cT_e$. We have by scaling
\cite{Verfurth:1994:Adaptive},
\begin{equation}
C \|\phi\|_{L_2(\cT_e)}\leq \left(\int_e h_e \sjump{\nabla u_h}^2
ds\right)^{1/2}, \label{eq:BBE}
\end{equation}
for some mesh independent constant $C$. Then, using
\eqref{eq:sigmadef}, Poincar\'e's inequality and a standard
inverse inequality \cite{BScott:2008:FEM,Ciarlet:1978:FEM}, we
find
\begin{eqnarray*}
\int_e \sjump{\nabla u_h}^2 ds &\leq & C  \int_e \sjump{\nabla
u_h}\zeta_1\zeta_2 ds
= C  \int_{\cT_e}  \nabla u_h \cdot \nabla \phi \,dx  \nonumber\\
&=& C\left(\int_{\cT_e}  \nabla (u_h-u) \cdot \nabla \phi \,dx +\int_{\cT_e} f\,\phi\,dx-\langle \sigma,\phi\rangle \right)\nonumber\\
&=& C\left(\int_{\cT_e}  \nabla (u_h-u) \cdot \nabla \phi \,dx +\int_{\cT_e} (f-\sigma_h)\,\phi\,dx-\langle \sigma-\sigma_h,\phi\rangle \right)\nonumber\\
&\leq & C \left(  h_e \|f-\sigma_h\|_{L_2(\cT_e)}+|u-u_h|_{H^1(\cT_e)}+\|\sigma-\sigma_h\|_{H^{-1}(\cT_e)} \right)\, |\phi|_{H^1(\cT_e)}\\
&\leq & C \left(  h_e
\|f-\sigma_h\|_{L_2(\cT_e)}+|u-u_h|_{H^1(\cT_e)}+\|\sigma-\sigma_h\|_{H^{-1}(\cT_e)}
\right)\,h_e^{-1}\|\phi\|_{L^2(\cT_e)}.
\end{eqnarray*}
Finally using \eqref{eq:BBE}, \eqref{eq:effsigma} and
\eqref{eq:oscf2}, we complete the proof of \eqref{eq:effjump}.
\end{proof}

\section{Numerical Experiments}\label{sec:Numerics}
In this section, we present some numerical experiments. Let
$\Omega$ be the unit disk and the obstacle function to be
$\chi=1-2r^2$, where $r=\sqrt{x^2+y^2}$. The load function $f$ is
taken to be
\begin{equation*}
f(r) := \left\{ \begin{array}{ll} 0 & \text{ if } r< r_0,\\\\
4r_0/r & \text{ if } r\geq r_0,
\end{array}\right.
\end{equation*}
in such a way that the solution $u$ takes the form
\begin{equation*}
u(r):= \left\{ \begin{array}{ll} 1-2r^2 & \text{ if } r < r_0,\\\\
4r_0(1-r) & \text{ if } r \geq r_0,
\end{array}\right.
\end{equation*}
where $r_0=(\sqrt{2}-1)/\sqrt{2}$.

In the experiment, we use the adaptive algorithm consisting of
four successive modules
\begin{equation*}
{\rm SOLVE}\rightarrow  {\rm ESTIMATE} \rightarrow {\rm
MARK}\rightarrow {\rm REFINE}
\end{equation*}
We use the primal-dual active set strategy
\cite{HK:2003:activeset} in the step SOLVE to solve the discrete
obstacle problem. The estimator derived in  Theorem
\ref{thm:zeroobstacle} is computed in the step ESTIMATE and then
the D\"orfler's bulk marking strategy \cite{Dorfler:1996:Afem}
with parameter $\theta=0.3$ has been used in the step MARK to mark
the elements for refinement. Using the newest vertex bisection
algorithm, we refine the mesh and obtain a new mesh. The
convergence history of errors and estimators is depicted in Figure
\ref{fig:ErrEst}. The figure illustrates the optimal order of
convergence as well as the reliability of the error estimator. The
efficiency index := estimator/error is computed and plotted in
Figure \ref{fig:EI}. The free boundary set and the contact set
have been captured by the error estimator efficiently see Figure
\ref{fig:mesh}. The experiment clearly illustrates the theoretical
results derived in the article.

\begin{figure}[t]
\begin{center}
\includegraphics*[width=8cm,height=5cm]{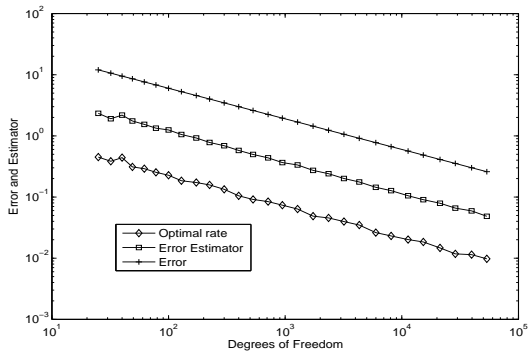}
\caption{{\large Error and Estimator}} \label{fig:ErrEst}
\end{center}
\end{figure}

\begin{figure}[t]
\begin{center}
\includegraphics*[width=8cm,height=5cm]{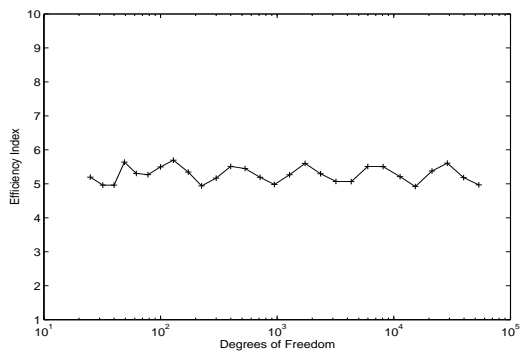}
\caption{{\large Efficiency Index}} \label{fig:EI}
\end{center}
\end{figure}

\begin{figure}[t]
\begin{center}
\includegraphics*[width=8cm,height=5cm]{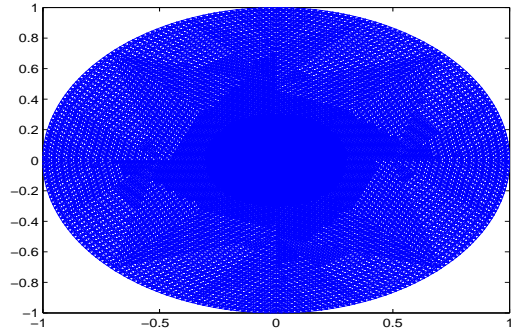}
\caption{{\large Mesh Refinement at Intermediate Level}}
\label{fig:mesh}
\end{center}
\end{figure}

\section{Conclusions}\label{sec:conclusions}
The a posteriori error analysis of finite element method for the
elliptic obstacle problem is revisited. By rewriting general
obstacle problem into a problem with zero obstacle and
inhomogeneous Dirichlet boundary condition, we have derived
reliable and efficient error bounds that are simpler and easily
computable. The analysis covers the inhomogeneous Dirichelt
boundary condition by constructing suitable post-processed
solutions that satisfy the boundary conditions exactly. We have
proposed two post-processing methods in which one is by harmonic
extension near the boundary and the other one is by linear
interpolation. These constructions are possible in three
dimensions with appropriate modifications and further the harmonic
extension is directly extendable. Numerical experiments illustrate
the theoretical results.

%\smallskip
%
%\par
%{\bf Acknowledgements:} The first author acknowledge the support
%from the UGC and the second author acknowledge the support from a
%DST Fast Track  project.
%%%%%%%%%%%%%%%%%%%%%%%%%%%%%%%%%%%%%%%%%%%%%%%%%%%%%%%%%%%%%%%%%
%


\begin{thebibliography}{10}
%
\bibitem{AO:2000:Book}
M.~Ainsworth and J.~T. Oden.
\newblock {\em A posteriori error estimation in finite element analysis}.
\newblock Pure and Applied Mathematics (New York). Wiley-Interscience [John
  Wiley \& Sons], New York, 2000.

\bibitem{AOL:1993:Apost}
M.~Ainsworth, J. T. oden and C. Y. Lee.
\newblock {\em Local a posteriori error estimates for variational inequalities}.
\newblock Numer. Meth. PDEs, 9:23-33, 1993.


\bibitem{AH:2009:VI}
K. Atkinson and W. Han.
\newblock {\em Theoretical Numerical Analysis. A functional analysis framework}.
Thrid edition, Springer, 2009.

\bibitem{BS:2014:hp-apost}
L. Banz and E. P. Stephan. A posteriori error estimates of
$hp$-adaptive IPDG-FEM for elliptic obstacle problems. {\em Appl.
Numer. Math}., 76:76--92, 2014.


\bibitem{BC:2004:VI}
S. Bartels and C. Carstensen.
\newblock Averaging techniques yield relaible a posteriori finite element error control for obstacle problems,
{\em Numer. Math.}. 99:225--249, 2004.

\bibitem{Belgacem:2000:SINUM}
F. Ben Belgacem. \newblock Numerical simulation of some
variational inequalities arisen from unilateral contact problems
by the finite element method. {\em SIAM J. Numer. Anal.},
37:1198–1216, 2000.

\bibitem{BS:2000:VI}
H. Blum and F. T. Suttmeier.
\newblock {An adaptive finite element discretization for a simplified Signorini problem}.
\newblock {\em Calcolo}, 37:65-77, 2000.

\bibitem{Braess:2005:VI}
D. Braess.
\newblock {A posteriori error estimators for obstacle problems-another look}.
\newblock {\em Numer. Math.}, 101:415-421, 2005.

\bibitem{BScott:2008:FEM}
S.C. Brenner and L.R. Scott.
\newblock {\em {The Mathematical Theory of Finite Element Methods $($Third
  Edition$)$}}.
\newblock Springer-Verlag, New York, 2008.

\bibitem{BLY:2012:C0IP}
S. C. Brenner, L. Sung and Y. Zhang. \newblock Finite element
methods for the displacement obstacle problem of clamped plates.
{\em Math. Comp.}, 81: 1247--1262, 2012.

\bibitem{BLY:2012:C0IP1}
S. C. Brenner, L. Sung, H. Zhang and Y. Zhang. \newblock A
quadratic C 0 interior penalty method for the displacement
obstacle problem of clamped Kirchhoff plates. {\em SIAM J. Numer.
Anal.,} 50: 3329--3350, 2012.

\bibitem{BHR:1977:VI}
F. Brezzi, W. W. Hager, and P. A. Raviart.
\newblock {Error estimates for the finite element solution of variational inequalities}, Part I. Primal theory.
\newblock {\em Numer. Math.}, 28:431--443, 1977.

\bibitem{BCH:2007:VI}
D. Braess, C. Carstensen and R.H.W. Hoppe.
\newblock {Convergence analysis of a conforming adaptive finite element method for an obstacle problem}.
\newblock {\em Numer. Math.}, 107:455--471, 2007.

\bibitem{BCH:2009:AVI}
\newblock D. Braess, C. Carstensen and R. Hoppe.  Error reduction in
adaptive finite element approximations of elliptic obstacle
problems. {\em J. Comput. Math.}, 27:148--169, 2009.

\bibitem{CC:1996:Hhalf}
C. Carstensen. \newblock Efficiency of a posteriori BEM-error
estimates for first-kind integral equations on quasi-uniform
meshes. {\em Math. Comp.}, 65:69--84, 1996.


\bibitem{CN:2000:VI}
Z. Chen and R. Nochetto.
\newblock {Residual type a posteriori error estimates for elliptic obstacle problems}.
\newblock {\em Numer. Math.}, 84:527--548, 2000.


\bibitem{Ciarlet:1978:FEM}
P.G. Ciarlet.
\newblock {\em {The Finite Element Method for Elliptic Problems}}.
\newblock North-Holland, Amsterdam, 1978.

\bibitem{Dorfler:1996:Afem}
W. D\"orlfer.
\newblock A convergent adaptive algorithm for Poisson's equation.
\newblock {\em SIAM J. Numer. Anal.}, 33:1106--1124, 1996.

\bibitem{DH:2015:Signorini}
G. Drouet and P. Hild. Optimal convergence for discrete
variational inequalities modelling signorini contact in 2d and 3d
without additional assumptions on the unknown contact set. {\em
SIAM J. Numer. Anal.}, 53:1488--1507, 2015.


\bibitem{Falk:1974:VI}
R. S. Falk. \newblock Error estimates for the approximation of a
class of variational inequalities. {\em  Math. Comp}., 28:963-971,
(1974).

\bibitem{FPP:2014:Obstacle}
M. Feischl, M. Page and D. Praetorius. Convergence of adaptive FEM
for some elliptic obstacle problem with inhomogeneous Dirichlet
data. {\em Int. J. Numer. Anal. Model.}, 11:229--253, 2014.


\bibitem{Glowinski:2008:VI}
R. Glowinski.
\newblock {\em Numerical Methods for Nonlinear Variational
Problems}. Springer-Verlag, Berlin, 2008.

\bibitem{GKVZ:2011:VeeserHirarchy}
\newblock Q. Zou, A. Veeser, R. Kornhuber and C. Gr\"aser.
\newblock Hierarchical error estimates for the energy functional in obstacle
problems. \newblock {\em Numer. Math.}, 117:653--677, 2011.

\bibitem{TG:2014:VIDG}
T. Gudi and K. Porwal. \newblock A posteriori error control of
discontinuous Galerkin methods for elliptic obstacle problems.
{\em Math. Comput.}, 83:579--602, 2014.

\bibitem{TG:2014:VIDG1}
T. Gudi and K. Porwal. \newblock A remark on the a posteriori
error analysis of discontinuous Galerkin methods for obstacle
problem. {\em Comput. Meth. Appl. Math.}, 14:71--87, 2014.

\bibitem{TG:2015:VIP2}
T. Gudi and K. Porwal. \newblock A reliable residual based a
posteriori error estimator for a quadratic finite element method
for the elliptic obstacle problem. {\em Comput. Methods Appl.
Math.}, 15:145--160, 2015.

\bibitem{Gwinner:2009:pfem}
J. Gwinner. On the $p$-version approximation in the boundary
element method for a variational inequality of the second kind
modelling unilateral contact and given friction. {\em Appl. Numer.
Math.}, 59:2774--2784, 2009.

\bibitem{Gwinner:2013:pfem}
J. Gwinner. $hp$ - FEM convergence for unilateral contact problems
with Tresca friction in plane linear elastostatics. {\em J.
Comput. Appl. Math.}, 254:175--184, 2013.


\bibitem{HR:2012:Signorini}
P. Hild and Y. Renard.
\newblock {An improved a priori error analysis for finite
element approximations of signorini's problem}.
\newblock {\em SIAM J. Numer. Anal.}, 50:2400--2419, 2012.

\bibitem{HN:2005:Signorini}
P. Hild and S. Nicaise.
\newblock {A posteriori error estimates of residual type for Signorini's problem}.
\newblock {\em Numer. Math}, 101:523-549, 2005.


\bibitem{HK:2003:activeset}
M. Hinterm\"uller, K. Ito and K. Kunish.
\newblock{The primal-dual active set strategy as a semismooth Newton
method.}
\newblock {\em SIAM J. Optim.}, 13:865--888, 2003.

\bibitem{HK:1994:multiadaptive}
R. H. W. Hoppe and R. Kornhuber.
\newblock{Adaptive multilevel methods for obstacle problems}.
\newblock {\em SIAM J. Numer. Anal.}, 31:301--323, 1994.


\bibitem{HW:2005:Signorini}
S. H\"ueber and B. I. Wohlmuth.
\newblock {An optimal a priori error estimate for nonlinear
multibody contact problems}.
\newblock {\em SIAM J. Numer. Anal.}, 43:156--173, 2005.

\bibitem{Kesavan:book}
S. Kesavan. \newblock {\em  Topics in Functional Analysis and
Applications}.
\newblock New Age International Ltd, Publishers, New Delhi, 1989.


\bibitem{KS:2000:VI}
D. Kinderlehrer and G. Stampacchia. \newblock {\em  An
Introduction to Variational Inequalities and Their Applications}.
\newblock SIAM, Philadelphia, 2000.


\bibitem{Royden:book}
H. L. Royden. \newblock {\em  Real Analysis}.
\newblock Third edition. Prnetice-Hall of India Pvt. Ltd, New Delhi 2008.

\bibitem{NSV:2005:Apost}
\newblock R. Nochetto, K. Siebert, and A. Veeser.
Fully localized a posteriori error estimators and barrier sets for
contact problems. {\em SIAM J. Numer. Anal.}, 42:2118--2135, 2005.

\bibitem{NSV:2003:Local}
\newblock R. Nochetto, K. Siebert, and A. Veeser.
Pointwise a posteriori error control for elliptic obstacle
problems. {\em Numer. Math.},  95:163--195, 2003.

\bibitem{NPZ:2010:VI}
R. Nochetto, T. V. Petersdorff and C. S. Zhang.
\newblock {A posteriori error analysis for a class of integral equations and variational inequalities}.
\newblock {\em Numer. Math}, 116:519--552, 2010.


\bibitem{PP:2013:AconvVI}
\newblock M. Page and D. Praetorius. Convergence of adaptive FEM for some elliptic obstacle problem. {\em Appl. Anal.},
92:595--615, 2013.

\bibitem{AS:2011:Contact}
A. Schr\"oder. \newblock{Mixed finite element methods of
higher-order for model contact problems}, SIAM J. Numer. Anal.,
49:2323--2339, 2011.

\bibitem{SV:2007:APost}
\newblock K. Siebert and A. Veeser. A uniliterally constrained quadratic
minimization with adaptive finite elements. {\em SIAM. J. Optim.},
18:260--289, 2007.

\bibitem{Sutt:book}
F. T. Suttmeier.
\newblock {\em Numerical solution of Variational Inequalities by Adaptive
Finite Elements}. Vieweg+Teubner Research, 2008.

\bibitem{Veeser:2001:VI}
A. Veeser.
\newblock {Efficient and Relaible a posteriori error estimators for elliptic obstacle problems}.
\newblock {\em SIAM J. Numer. Anal.}, 39:146--167, 2001.

\bibitem{Verfurth:1994:Adaptive}
R.~Verf{\"u}rth.
\newblock A posteriori error estimation and adaptive mesh-refinement
  techniques.
\newblock In {\em Proceedings of the Fifth International Congress on
  Computational and Applied Mathematics (Leuven, 1992)}, volume~50, pages
  67--83, 1994.

\bibitem{Verfurth:1995:AdaptiveBook}
R.~{Verf\"urth}.
\newblock {\em A Review of A Posteriori Error Estmation and Adaptive
  Mesh-Refinement Techniques}.
\newblock Wiley-Teubner, Chichester, 1995.


\bibitem{Wang:2002:P2VI}
L. Wang.
\newblock {On the quadratic finite element approximation to the obstacle problem.}
\newblock {\em Numer. Math.}, 92:771--778, 2002.

\bibitem{WHC:2010:DGVI}
F. Wang, W. Han and X.Cheng.
\newblock {Discontinuous Galerkin methods for solving elliptic variational inequalities}.
\newblock {\em SIAM J. Numer. Anal.}, 48:708--733, 2010.

\bibitem{WHE:2015:ApostDG}
F. Wang, W. Han, and J. Eichholz and X. Cheng. \newblock A
posteriori error estimates for discontinuous Galerkin methods of
obstacle problems. {\em Nonlinear Anal. Real World Appl.},
22:664--679, 2015.

\bibitem{WW:2010:Apost}
A. Weiss and B. I. Wohlmuth.
\newblock {A posteriori error estimator for obstacle problems}.
\newblock {\em SIAM J. Numer. Anal.}, 32:2627--2658, 2010.

\bibitem{WW:2009:Contact}
A. Wiess and B. I. Wohlmuth.
\newblock {A posteriori error estimator and error control for contact problems}.
\newblock {\em Math. Comp.}, 78:1237-1267, 2009.


\end{thebibliography}
\end{document}